\documentclass[12pt]{amsart}
\usepackage{amssymb,bm,fullpage,graphicx,hhline,longtable,mathtools,mathrsfs,tikz,tikz-cd,upref,xurl}
\usepackage{anyfontsize}
\usepackage{newtxtext,newtxmath}
\usepackage{hyperref}
\usepackage{diagbox}
\allowdisplaybreaks
\newtheorem{theorem}{Theorem}[section]

\newtheorem{cor}[theorem]{Corollary}
\newtheorem{lemma}[theorem]{Lemma}

\usepackage{bbm}
\theoremstyle{definition}

\theoremstyle{remark}

\usepackage{graphicx}
\numberwithin{equation}{section}

\let\originalleft\left
\let\originalright\right
\renewcommand{\left}{\mathopen{}\mathclose\bgroup\originalleft}
\renewcommand{\right}{\aftergroup\egroup\originalright}

\makeatletter
\@namedef{subjclassname@2020}{%
  \textup{2020} Mathematics Subject Classification}
\makeatother

\begin{document}
\title{Strong $q$-analogues for values of the Dirichlet beta function}
%Asymptotics, CM evaluations and $q$-analogues for values of the Dirichlet beta function at positive odd integers
\author{Ankush Goswami, Timothy Huber}
\address{School of Mathematical and Statistical Sciences, the University of Texas Rio Grande Valley, 1201 W. University Dr., Edinburg, TX, 78539}
\email{ankush.goswami@utrgv.edu, ankushgoswami3@gmail.com}
\email{timothy.huber@utrgv.edu}

\maketitle
\begin{abstract}
An infinite class of relations between modular forms is constructed that generalizes evaluations of the Dirichlet beta function at odd positive integers. The work is motivated by a base case appearing in Ramanujan's Notebooks and a parallel construction for the Riemann zeta function. The identities are shown to be strong $q$-analogues by virtue of their reduction to the classical beta evaluations as $q\to 1^{-}$ and explicit evaluations at CM points for $|q|<1$. Inequalities of Deligne determine asymptotic formulas for the Fourier coefficients of the associated modular forms.  
% By studying a certain space of modular forms, we obtain $q$-series generalizations of the evaluations values of the evaluations Dirichlet beta function at positive odd integers. These $q$-analogues are evaluated at CM points. We also obtain asymptotics of a certain counting function.     
\end{abstract}

\section{Introduction}
\noindent For $|q|<1$, define $(a;q)_{\infty} = \prod_{n=0}^{\infty} (1-aq^{n})$. The two identities \cite{Sun} 
\begin{align} \label{sun}
    \sum_{k=0}^{\infty} \frac{q^k(1+q^{2k+1})}{(1-q^{2k+1})^{2}} = \frac{(q^{2}; q^{2})_{\infty}^{4}}{(q;q^{2})_{\infty}^{4}}, 
    % \prod_{n=1}^{\infty} \frac{(1-q^{2n})^{4}}{(1-q^{2n-1})^{4}},
    \qquad \sum_{k=0}^\infty\dfrac{q^k(1+q^{2k+1})}{(1-q^{2k+1})^2}=\dfrac{(q^2;q^2)_\infty^4}{(q;q^2)_\infty^4}.
    % \sum_{k=0}^{\infty} \frac{q^{2k-\lfloor (-1)^{k} k/2 \rfloor}}{(1-q^{2k+1})^{2}} = \frac{(q^{2}; q^{2})_{\infty}^{2} (q^{4}; q^{4})_{\infty}^{2} }{(q; q^{2})_{\infty}^{2}(q^{2}; q^{4})_{\infty}^{2}}   
    % % \prod_{n=1}^{\infty} \frac{(1-q^{2n})^{2} (1-q^{4n})^{2}}{(1-q^{2n-1})^{2} (1-q^{4n-2})^{2}}.
\end{align}
generalize Euler's evaluations in the cases $k=1,2$ for 
\begin{equation}\label{zeta}
 \zeta(2k)=\dfrac{(-1)^{k+1}B_{2k}2^{2k}}{2(2k)!}\pi^{2k}, \qquad \frac{z}{e^{z} -1} := \sum_{n=0}^{\infty} \frac{B_{n} z^{n}}{n!},
\end{equation}
where $\zeta(s) =\sum_{n\geq 1} 1/n^{s}$, $\operatorname{Re}(s)>1$ is the Riemann zeta function. 
After multiplication by a  finite factor, the identities \eqref{sun}, respectively, tend to evaluations of $\zeta(2)$ and $\zeta(4)$ as $q\to 1$. The first author \cite{AG0} used the theory of modular forms to determine explicit $q$-expansions generalizing every case of \eqref{zeta}.
Ono and Dawsey \cite{DO} employed the theory of complex multiplication (CM) to obtain further specializations of Goswami and Sun's series for $|q|<1$. Their evaluation at other points in the unit circle were given in terms of algebraic multiples of $\Gamma$-values at CM points. 

The Riemann zeta function is closely aligned with the Dirichlet beta function, defined by
\begin{equation}\label{betadef}
\beta(s)=\sum_{n\geq 1}\dfrac{\chi_{-4}(n)}{n^s}=\sum_{n\geq 0}\dfrac{(-1)^n}{(2n+1)^s}, \qquad \operatorname{Re}(s)>0, \quad \chi_{-4}=\left(\frac{-4}{\cdot}\right).
\end{equation}
Both $\zeta(s)$ and $\beta(s)$ have meromorphic and analytic continuations, respectively, to the complex plane, and their Euler products imply that the Riemann hypothesis for each is equivalent. In \cite{KRZ, Z}, arithmetic and diophantine properties of $q$-analogues of values of $\zeta(s)$ were studied.
% Therefore, $q$-series generalizations of $\zeta$ and $\beta$ can be used to study the associated $\Gamma$-functions.
This work formulates identities between modular forms that generalize odd evaluations of $\beta(s)$
\begin{equation}\label{betagenvalues}
\beta(2k+1)=\dfrac{(-1)^k E_{2k}}{4^{k+1}(2k)!}\pi^{2k+1}, \qquad \frac{2}{e^{z}+e^{-z}} := \sum_{n=0}^{\infty} \frac{E_{n} z^{n}}{n!}. \end{equation}
We show the resulting infinite class of of identities may also be evaluated at CM points to obtain additional computational information involving $\Gamma$-values. 
% We derive an infinite class of identities between modular forms whose limiting values as $q \to 1$ tend to \eqref{betagenvalues}. 
The first three generalizations of \eqref{betagenvalues} are
\begin{equation}\label{qanarama}
\sum_{n=0}^{\infty}(-1)^n\dfrac{q^n}{1-q^{2n+1}}=
%\prod_{n=1}^\infty\dfrac{(1-q^{4n})^2}{(1-q^{4n-2})^2},
 \dfrac{(q^4;q^4)_\infty^2}{(q^2;q^4)_\infty^2}, \qquad \sum_{n=0}^{\infty} (-1)^n\dfrac{q^{2n}(1+q^{2n+1})}{(1-q^{2n+1})^3}
 %=\prod_{n=1}^\infty\dfrac{(1-q^{4n-2})^2(1-q^{4n})^6}{(1-q^{2n-1})^4}.
=\dfrac{(q^2;q^4)_\infty^2(q^4;q^4)_\infty^6}{(q;q^2)_\infty^4},
\end{equation}
\begin{align} 
 \nonumber   \sum_{n=1}^{\infty} &\frac{\chi_{-4}(n)q^{2n}\left(1+57q^{2n}+302q^{4n}+302q^{6n}+57q^{8n}+q^{10n}\right)}{\left(1-q^{2n}\right)^{7}} +\frac{17q^{4}\left(q^{8};q^{8}\right)_{\infty}^{4}\left(q^{4};q^{4}\right)_{\infty}^{22}}{\left(q^{2};q^{2}\right)_{\infty}^{12}}\\ &=61\cdot\frac{q^{4}\left(q^{4};q^{4}\right)_{\infty}^{22}\left(q^{8};q^{8}\right)_{\infty}^{4}}{\left(q^{2};q^{2}\right)_{\infty}^{12}}+\frac{q^{2}\left(q^{4};q^{4}\right)_{\infty}^{46}}{\left(q^{2};q^{2}\right)_{\infty}^{20}\left(q^{8};q^{8}\right)_{\infty}^{12}} +16\cdot\frac{q^{6}\left(q^{8};q^{8}\right)_{\infty}^{20}}{\left(q^{2};q^{2}\right)_{\infty}^{4}\left(q^{4};q^{4}\right)_{\infty}^{2}}. \label{k=3}
\end{align}
Identities \eqref{qanarama} are due to Ramanujan \cite[Example (iv), p. 139]{B1} and Hou-Sun \cite{HS}, respectively, and reduce to $\beta(1)=\pi/4$ and $\beta(3)=\pi^3/32$. A different identity generalizing $\beta(1)=\pi/4$ was obtained in \cite{HKS}. Equation \eqref{k=3} corresponds to $\beta(5)=5 \pi ^5/1536$ and is representative of the class identities derived in the present work.

The infinite class of identities considered here generalizing \eqref{betagenvalues} permit much more than a specialization of the L-function since they are valid in a larger domain. Because the generalizations involve a parameter $q$ inside the unit disk, evaluations of the series in terms of a variety of $\Gamma$-values may be obtained by restricting the series to appropriate values. 
% $q$-analogues play an important role in solutions to many-body problems, permitting the evaluation of exact results for arbitrary model parameters \cite{SunSinitsyn2017}. In these models, the limit $q\to 1^{-}$ usually corresponds to classical cases, while evaluations for $|q|<1$ correspond to more specialized scenarios. 
Ono and Dawsey term such identities \textit{strong q-analogues} of the beta evaluations because one obtains evaluations corresponding to CM points for $|q|<1$ as well as limiting values as $q\to 1^{-}$. Such $q$-analogues appear prominently in mathematical physics, generalizing classical formulas through the parameter $q$ and recovering the original formulas under the limit $q\to 1^{-}$ \cite{PhysRevA.94.033808}. 

The $q$-expansions derived here are constructed by writing relevant Lambert series in terms of Eisenstein series and cuspidal quotients of the Dedekind eta function $\eta(\tau)=q^{1/24}(q;q)_\infty$, $q = e^{2 \pi i\tau}$. 
We will assume the reader is familiar with fundamental aspects of the theory of modular functions and modular forms appearing in \cite{DS,Koblitz,Ono}. To state our main results, define $$\Gamma_0(N):=\left\{\begin{pmatrix} a&b\\c&d
    \end{pmatrix}\in SL_2(\mathbb{Z}): c\equiv 0 \pmod{N}\right\}.$$ Denote by $M_k(N,\chi)$ and $S_k(N,\chi)$ the $\mathbb{C}$-vector spaces of modular and cusp forms, respectively, of weight $k$ over $\Gamma_0(N)$ and denote the space of Eisenstein series by $E_k(N,\chi)$. The theory of eta quotients \cite{Ono} (cf. Theorem \ref{GHN}) may be used to show that 
\begin{equation}
f_k(\tau):=\dfrac{\eta(4\tau)^{8k-2}\eta(8\tau)^4}{\eta(2\tau)^{4k}}\in M_{2k+1}(8,\chi_{-4}),\quad k\in\mathbb{Z}_{\geq 0},    
\end{equation}
where $\chi_{-4}$ is defined as in \eqref{betadef}. We designate Eisenstein series defined by Dirichlet characters $\chi$ and $\psi$ of modulus $L$ and $R$, respectively, by
\begin{equation*}
M_{k,\chi,\psi}(\tau):=\sideset{}{'}\sum_{m,n=-\infty}^\infty\dfrac{\chi(m)\psi(n)}{(m\tau+n)^k},\quad \tau\in\mathbb{H},  
\end{equation*}
where $\Sigma'$ denotes summation over all pairs of integers except $(0,0)$. Let $\psi^{0}$ be the primitive character associated to $\psi$, and let $W(\psi^0)$ be the Gauss sum for the character $\psi^0$, defined by 
\begin{equation} \label{akdef}
W(\psi^0)=\displaystyle\sum_{a=0}^{r_\psi-1}\psi^0(a)e^{2\pi i a/M}\quad \text{and}\quad \mathcal{A}_k=\dfrac{2(-2\pi i)^{k}W(\psi^0)}{R^k(k-1)!},
\end{equation}
where $r_\psi$ is the conductor of $\psi$.
%(or modulus of $\psi^0$). 
We will use the Fourier series expansion of $M_{k,\chi,\psi}(\tau)$ \cite{Miyake}, in the case that $\chi(-1)\psi(-1)=(-1)^k$ given by
\begin{equation} \label{meis}
M_{k,\chi,\psi}(\tau)=c_k+\mathcal{A}_k\sum_{n\geq 1}a(n)q^{n/R}, \quad \text{where} \quad c_k=\begin{cases}
2L(k,\psi),&\chi:\;\text{principal},\\
0,&\text{otherwise},
\end{cases}\end{equation}
and 
\begin{equation*}
a(n)=\sum_{d\mid n}\chi(n/d)d^{k-1}\sum_{c\mid (\ell,d)}c\mu(\ell/c)\psi^0(\ell/c)\overline{\psi^0}(d/c),   
\end{equation*}
and where $\ell=R/r_\psi$, and $\mu$ denotes the M\"obius function.
% and for a Dirichlet character $\varepsilon$ of modulus $N$ over $\mathbb{C}$, $B_{k,\varepsilon}$ is the generalized Bernoulli number attached to $\varepsilon$, defined as 
% \begin{equation*}
% \sum_{a=1}^N\dfrac{\varepsilon(a)\cdot x\cdot e^{ax}}{e^{Nx}-1}=\sum_{k=0}^\infty B_{k,\epsilon}\cdot\dfrac{x^k}{k!}.    
% \end{equation*}
% The case $\varepsilon=1$ (trivial character with modulus $N=1$) above recovers the classical Bernoulli numbers with $B_{1,\varepsilon}=-B_1=1/2$. 
We will employ a normalized version of the Fourier expansion
\begin{equation}\label{Eisensteinseries}
\mathscr{E}_{k,\chi,\psi}(\tau)=\dfrac{c_k}{\mathcal{A}_k}+\sum_{n\geq 1}a(n)q^{n/R}.    
\end{equation}

With the above notation, our main results may be stated in the next theorems. 
\begin{theorem}\label{main1}
Define $E_{k}$ by \eqref{betagenvalues}. For $k\geq 1$, we have
\begin{align*}
f_k(\tau)-H_k(\tau)\in S_{2k+1}(8,\chi_{-4}),    
\end{align*}
where 
\begin{align*}
H_k(\tau):=\begin{cases}
\dfrac{1}{2^{2k}E_{2k}}\cdot \left(2^{2k}\cdot \mathscr{E}_{2k+1,\chi_{-4},1}(2\tau)-\mathscr{E}_{2k+1,\chi_{-4},\chi_2}(2\tau)\right),&k\equiv 0\pmod{2},\\
-\dfrac{1}{E_{2k}}\cdot \mathscr{E}_{2k+1,\chi_{-4},1}(2\tau),&k\equiv 1\pmod{2}.
\end{cases}    
\end{align*}
\end{theorem}
\begin{theorem}\label{main2}
Let $f_{k}$ be defined as in Theorem \ref{main1}. Denote the Eulerian polynomials $A_{k}$ by
\begin{align*}
    \sum_{k=0}^{\infty} A_{k}(x)\frac{z^{n}}{n!} := \frac{x-1}{x-e^{(x-1)z}}.
\end{align*}
Then, for $k\geq 1$, 
\begin{align}\label{qana}
f_k(\tau)=\begin{cases}
\dfrac{1}{2^{2k}\cdot E_{2k}}\displaystyle\sum_{n\geq 1}\chi_{-4}(n)\dfrac{q^{n}P_{k}(q^{n})}{(1-q^{2n})^{2k+1}},&k\equiv 0\pmod{2},\\
-\dfrac{1}{E_{2k}}\displaystyle\sum_{n\geq 1}\chi_{-4}(n)\dfrac{q^{2n} A_{2k}(q^{2n})}{(1-q^{2n})^{2k+1}},&k\equiv 1\pmod{2},    
\end{cases}+T_{2k+1}(\tau)    
\end{align}
where $T_{2k+1}(\tau)\in S_{2k+1}(8,\chi_{-4})$ and $P_{k}(t)\in\mathbb{Z}[t]$ is of degree $4k$, and defined by 
\begin{equation*}
P_k(t):=(1+t)^{2k+1}\cdot A_{2k}(t)-2^{2k}\cdot t\cdot A_{2k}(t^2).  
\end{equation*}
If both sides of \eqref{qana} are multiplied by $(1-q)^{2k+1}$ and $q\to 1^{-}$, then \eqref{qana} specializes to \eqref{betagenvalues}. 
% Thus, \eqref{qana} is a $q$-analogue of \eqref{betagenvalues}.
\end{theorem}

Theorem \ref{main1} provides a canonical decomposition for the eta quotient $f_{k}$ into Eisenstein series and eta quotients that are cusp forms. Theorem \ref{main2} reformulates the expansions of Theorem \ref{main1} as $q$-analogues of the evaluation of the beta function at odd positive values. The cusp component $T_{k}(\tau)$ in Theorem \ref{main2} may be calculated explicitly in terms of the eta quotients generating a basis
\begin{equation*}
F(\tau):=\dfrac{\eta(4\tau)^8}{\eta(2\tau)^4},\quad \theta(\tau):=\sum_{n=-\infty}^\infty q^{n^2}=\dfrac{\eta(2\tau)^5}{\eta(\tau)^2\eta(4\tau)^2},\quad F_2(\tau)=F(2\tau),\quad \theta_2(\tau)=\theta(2\tau).    
\end{equation*}
% that are shown in Theorem \ref{main3} generate the graded algebra
% \begin{equation*}
% M_{2k+1}(8,\chi_{-4})\cong \mathbb{C}[F,\theta,F_2,\theta_2]\oplus\mathbb{C}\left[\dfrac{FF_2^k}{\theta_2^2}\right], \qquad k\ge 1.
%\end{equation*}
\begin{theorem}\label{main3}
% As a graded algebra, we have for $k\geq 1$
% \begin{equation*}
% % M_{2k+1}(8,\chi_{-4})\cong \mathbb{C}[F,\theta,F_2,\theta_2]\oplus\mathbb{C}\left[\dfrac{FF_2^k}{\theta_2^2}\right].    
% \end{equation*}
The set
\begin{equation*}
\left\{F^{\ell}\theta_2^{4(k-\ell)+2}: 0\leq \ell\leq k-1\right\}\bigcup\left\{F^{2k-\ell}F_2^{\ell-k}\theta_2^{2}: k\leq \ell\leq 2k\right\}\bigcup\left\{\dfrac{FF_2^{k}}{\theta_2^2}\right\}    
\end{equation*}
forms a basis for $M_{2k+1}(8,\chi_{-4})$. That is, every $g\in M_{2k+1}(8,\chi_{-4})$ can be uniquely written as
\begin{align*}
g(\tau)=\sum_{\ell=0}^{k-1}\alpha_k(\ell)\cdot F(\tau)^\ell\theta(2\tau)^{4(k-\ell)+2}+\sum_{\ell=k}^{2k}\beta_k(\ell)\cdot F(\tau)^{2k-\ell}F(2\tau)^{\ell-k}\theta(2\tau)^2+\gamma_k\cdot \dfrac{F(\tau)F(2\tau)^k}{\theta(2\tau)^2},  
\end{align*}
where $\alpha_k(r), \beta_k(s)\in\mathbb{C}$ with $0\leq r\leq k-1,\;k\leq s\leq 2k$. Furthermore, if $g(\tau)$ is a cusp form, then the coefficients in the linear combination above satisfy
\begin{align*}
    &(1) \quad \alpha_k(0)=0 = \gamma_{k}, \qquad (2) \quad \displaystyle\sum_{\ell=0}^{k-1}\dfrac{\alpha_k(\ell)}{4^\ell}+\sum_{\ell=k}^{2k}\dfrac{\beta_k(\ell)}{4^\ell}+\dfrac{\gamma_k}{2^{4k+3}}=0, \\ 
    &(3) \quad \displaystyle\sum_{\ell=0}^{k-1}(-1)^\ell\dfrac{\alpha_k(\ell)}{4^{\ell}}+\sum_{\ell=k}^{2k}(-1)^\ell\dfrac{\beta_k(\ell)}{4^\ell}-\dfrac{\gamma_k}{2^{4k+2}}=0.
\end{align*}

% \begin{enumerate}
%     \item $\alpha_k(0)=0 = \gamma_{k}$,
%     % \item $\dfrac{i(-1)^{k+1}}{2^{4k+2}}\displaystyle\sum_{\ell=0}^{k-1}\dfrac{\alpha_k(\ell)}{4^\ell}+\dfrac{i(-1)^{k+1}}{2^{4k+2}}\sum_{\ell=k}^{2k}\dfrac{\beta_k(\ell)}{4^\ell}+\dfrac{i(-1)^{k+1}\gamma_k}{2^{8k+5}}=0$
%     \item $\displaystyle\sum_{\ell=0}^{k-1}\dfrac{\alpha_k(\ell)}{4^\ell}+\sum_{\ell=k}^{2k}\dfrac{\beta_k(\ell)}{4^\ell}+\dfrac{\gamma_k}{2^{4k+3}}=0$
%     % \item $\displaystyle\dfrac{i(-1)^{k+1}}{2^{2k+1}}\sum_{\ell=0}^{k-1}(-1)^\ell\dfrac{\alpha_k(\ell)}{4^{\ell}}+\dfrac{i(-1)^{k+1}}{2^{2k+1}}\sum_{\ell=k}^{2k}(-1)^\ell\dfrac{\beta_k(\ell)}{4^\ell}+\dfrac{i(-1)^k}{2^{6k+3}}\gamma_k=0$
%      \item $\displaystyle\sum_{\ell=0}^{k-1}(-1)^\ell\dfrac{\alpha_k(\ell)}{4^{\ell}}+\sum_{\ell=k}^{2k}(-1)^\ell\dfrac{\beta_k(\ell)}{4^\ell}-\dfrac{\gamma_k}{2^{4k+2}}=0$
% %     \item $\gamma_k=0$.
% \end{enumerate}    
\end{theorem}

These results lead to evaluations for the Lambert series on the left side of \eqref{qana} at CM points. 
\begin{theorem}\label{main4}
Let $H_{k}$ be defined as in Theorem \ref{main1}. For $r\in\mathbb{N}$, consider the CM point $\tau_r:=2^r i$. Then we have
\begin{align*}
H_k(\tau_r)&=\dfrac{a^{4k+2}}{2^{(2k+1)(r+3)}}\cdot \left(\dfrac{e^{-2^{r+1}(k+1)\pi}}{4^{k+1}}+\sum_{\ell=1}^{k-1}\dfrac{\alpha_k(\ell)}{4^\ell}\cdot e^{M(\ell,r)}+\sum_{\ell=k}^{2k}\dfrac{\beta_k(\ell)}{4^\ell}\cdot e^{M(\ell,r)}\right)\cdot e^{J(k,r)},
\end{align*}
where
\begin{align*}
J(k,r)&:=\frac{\pi(2k+1)}{3}-\frac{2k+1}{\pi}\sum_{m=1}^{r+1}2^m L_m-\frac{2^{r+2}(4k+1)}{\pi}L_{r+2}-\frac{2^{r+4}}{\pi}L_{r+3},\notag\\
M(\ell,r)&:=-\frac{\ell\cdot 2^{r+1}}{\pi}\sum_{m=1}^{r+1}2^m L_m-\frac{2^{r+3}(k-\ell+1)}{\pi}L_{r+2}+\frac{2^{r+5}(k-\ell+1)}{\pi}L_{r+3},
\end{align*}
and 
where $L_\ell$ is defined by convergent double series  
\begin{align*}
L_\ell:=\sum_{m,n\geq 1}\dfrac{(-1)^m}{n^2+2^\ell m^2}.   
\end{align*}
In particular,
\begin{align} \label{ev}
\sum_{m,n\geq 1}\dfrac{(-1)^m}{n^2+2^\ell m^2}=\begin{cases}
-\dfrac{\pi^2}{24}-\dfrac{\pi\cdot\log 2}{8},&\ell=1,\\ -\dfrac{7\pi^2}{96}-\dfrac{\pi\cdot\log 2}{32}-\dfrac{\pi\cdot\log(\sqrt{2}-1)}{8},&\ell=2,\\
-\dfrac{31\pi^2}{384}-\dfrac{5\pi\log 2}{128}+\dfrac{\pi\cdot\log(\sqrt{2}-1)}{32}-\dfrac{\pi\cdot\log(1-2^{-1/4})}{8},&\ell=3.
\end{cases}    
\end{align}
% If $D<0$ is a fundamental discriminant of the imaginary quadratic field and $\tau\in\mathbb{H}\cap\mathbb{Q}(\sqrt{D})$. Then for all $k\geq 1$, we have
% \begin{equation*}
% H_k(\tau)\in\overline{\mathbb{Q}}\cdot \omega_D^{2k+1}    
% \end{equation*}
% where $\omega_D$ is defined as
% \begin{equation*}
% \omega_D=\dfrac{1}{\sqrt{\pi}}\left(\prod_{j=1}^{|D|-1}\Gamma\left(\dfrac{j}{|D|}\right)^{\chi_D(j)}\right)^{\frac{1}{h'(D)}}     
% \end{equation*}
% and where we have used the same notations and conventions as in section \ref{CMeva}.
\end{theorem}

In Section \ref{s2}, a proof of Theorem \ref{main1} is given by using evaluations for the appropriate modular forms at cusps. Section \ref{s3} is devoted to a proof of Theorem \ref{main2} and includes a derivation of the limiting values as $q\to 1^{-}$. In Section \ref{s4}, we develop a basis for the modular forms in question. This uniquely determines the eta quotients appearing in the strong $q$-analogues of \eqref{betagenvalues}. Section \ref{s5} addresses the CM evaluations and includes a proof the Theorem \ref{main4}. Section \ref{s6} concludes with an asymptotic formula for the coefficients given in Corollary \ref{cor1}.

\section{Proof of Theorem \ref{main1}} \label{s2}
%Note that $B_{2k+1,1}=B_{2k+1,\chi_2}=0$. 
We begin with a fundamental result for determining the modularity of certain quotients of the Dedekind eta function considered in this work \cite{GH,Ono}. 
\begin{theorem}[Gordon, Hughes and Newman]\label{GHN}
Let $f(\tau)=\prod_{\delta\mid N}\eta(\delta\tau)^{r_\delta}$ be an eta-quotient where $r_\delta\in\mathbb{Z}$ and $k=\frac{1}{2}\sum_{\delta\mid N}r_\delta$. Let $\alpha=r/s$ be any cusp representative with $r, s\in\mathbb{N}$, and $s\mid N$. If $f(\tau)$ satisfies the following properties:
     \begin{align*} &(1) \quad \displaystyle\sum_{\delta\mid N}\delta r_\delta\equiv 0\pmod{24},\qquad 
    (2) \quad \displaystyle\sum_{\delta\mid N}\dfrac{N}{\delta} r_\delta\equiv 0\pmod{24}, \\
    &(3) \quad \displaystyle\dfrac{N}{24}\sum_{\delta\mid N}\dfrac{(s,\delta)^2r_\delta}{(s,N/s)s\delta}\geq 0. \end{align*}
Then $f(\tau)$ is a holomorphic modular form on $\Gamma_0(N)$ with Nebentypus $\chi=\left(\frac{(-1)^k\prod_{\delta\mid N}\delta^{r_\delta}}{\cdot}\right)$ satisfying 
\begin{equation*}
(f|_k\gamma)(\tau)=\chi(d)f(\gamma\tau),\quad \gamma=\begin{pmatrix}
a&b\\c&d    
\end{pmatrix}\in\Gamma_0(N).    
\end{equation*}
If the inequality in (3) is strict, $f(\tau)$ is a cusp form.
\end{theorem}
Thus $f_k(\tau)$ is a holomorphic modular form of weight $2k+1$ over $\Gamma_0(8)$ and Nebentypus $\chi_{-4}$. Since $\text{dim}(M_{2k+1}(8,\chi_{-4}))=2k+2$, and $\text{dim}(S_{2k+1}(8,\chi_{-4}))=2k-2$, the corresponding Eisenstein space has dimension $2$. With $H_k(\tau)$ as in Theorem \ref{main1}, we show that 
\begin{align}\label{cuspform}
f_k(\tau)-H_k(\tau)\in S_{2k+1}(8,\chi_{-4})    
\end{align}
by showing that $H_{k}$ represents the linear combination of the two Eisenstein series appearing in each expansion. The coefficients in the linear combination are computed by evaluating both $f_{k}$ and $H_{k}$ at the inequivalent cusps $\left\{0,\frac{1}{2},\frac{1}{4},i\infty\right\}$ determined by $\Gamma_{0}(8)$. 
First, we calculate the value of $f_k(\tau)$ at each of the cusps. Recall the following transformations for $\eta(\tau)$:
\begin{align}\label{dedtrans}
\eta\left(-\frac{1}{\tau}\right)=\sqrt{-i\tau}\eta(\tau),\quad\text{and}\quad \eta\left(\tau+\frac{1}{2}\right)=e^{\frac{2\pi i}{48}}\dfrac{\eta(2\tau)^3}{\eta(\tau)\eta(4\tau)}.    
\end{align}
Clearly, $f_k(i\infty)=0$. In order to compute the value of $f_k$ at the cusps $0, 1/2$ and $1/4$, we consider the matrices
\begin{equation}\label{matrices}
 S=\begin{pmatrix}
0&-1\\1&0    
\end{pmatrix}, \quad \gamma_2=\begin{pmatrix}
1&0\\2&1    
\end{pmatrix} \quad \text{and}\quad \gamma_4=\begin{pmatrix}
1&0\\4&1    
\end{pmatrix}.   
\end{equation}
The transformations in \eqref{dedtrans} imply
\begin{align} \label{feval1}
(f_k|_{2k+1}S)\left(\tau\right)
% \frac{\eta(-4/\tau)^{8k-2}\eta(-8/\tau)^4}{\eta(-2/\tau)^{4k}}=\sqrt{\dfrac{(-i\tau/4)^{8k-2}(-i\tau/8)^{4}}{(-i\tau/2)^{4k}}}\cdot \frac{\eta(\tau/4)^{8k-2}\eta(\tau/8)^4}{\eta(\tau/2)^{4k}}\notag\\
% &=\dfrac{(-i)^{2k+1}}{2^{6k+4}}\cdot\tau^{2k+1}\cdot\frac{\eta(\tau/4)^{8k-2}\eta(\tau/8)^4}{\eta(\tau/2)^{4k}}
&=\dfrac{(-i)^{2k+1}}{2^{6k+4}}\cdot\dfrac{(q_4;q_4)_\infty^{8k-2}(q_8;q_8)_\infty^4}{(q_2;q_2)_\infty^{4k}},\\
\label{feval2}
(f_k|_{2k+1}\gamma_2)\left(\tau\right)&=\dfrac{i}{2^{4k+3}}\cdot\dfrac{(q^2;q^2)_\infty^{20k-6}(\xi_4q_2;q_2)_\infty^4}{(q^4;q^4)_\infty^{8k-2}(q;q)_\infty^{8k-2}},\\
\label{feval3}
(f_k|_{2k+1}\gamma_4)\left(\tau\right)&=-\dfrac{e^{-\frac{\pi i k}{3}}}{2^{10}}\cdot q^k\dfrac{(q^4;q^4)_\infty^{8k+10}}{(q^2;q^2)_\infty^{4k+4}(q^8;q^8)_\infty^4}.
\end{align}
Therefore $(f\mid_{2k+1} S)(i\infty)=\dfrac{(-i)^{2k+1}}{2^{6k+4}}$, $(f\mid_{2k+1}\gamma_2)(i\infty)=\dfrac{i}{2^{4k+3}}$, and $(f\mid_{2k+1} \gamma_4)(i\infty)=0$. 

We next compute the values of the Eisenstein series $\mathscr{E}_{2k+1,\chi_{-4},1}(2\tau)$ and $\mathscr{E}_{2k+1,\chi_{-4},\chi_2}(2\tau)$ at each of the cusps.  This will involve the level $N$ Eisenstein series notation introduced in Koblitz \cite{Koblitz},
\begin{equation} \label{kobeis}
G_{\ell}^{(\underline{a},N)}(\tau):=\sideset{}{'}\sum_{\substack{m_1,m_2\in\mathbb{Z}\\m_1\equiv a_1\!\!\!\!\!\!\pmod{N}\\m_2\equiv a_2\!\!\!\!\!\!\pmod{N}}}\dfrac{1}{(m_1\tau+m_2)^{\ell}}    
\end{equation}
for $\ell, N\in\mathbb{N}$, and $\underline{a}=(a_1,a_2)$, and where the tuple is reduced modulo $N$.  
%where $'$ in the sum above means that the case $(m_1,m_2)=(0,0)$ is excluded if encountered. 
For $\gamma\in SL_2(\mathbb{Z})$, we have the following crucial property of $G_{\ell}^{(\underline{a},N)}(\tau)$ permitting a re-centering the Eisenstein series and allowing for calculation of the value at a cusp:
\begin{equation}\label{transKob}
(G_{\ell}^{(\underline{a}, N)}|_\ell\gamma)(\tau)=G_{\ell}^{(\underline{a}\gamma,N)}(\tau).    
\end{equation}
% The Fourier expansion of $G_{\ell}^{(\underline{a},N)}(\tau)$ is given by
% \begin{align}\label{Kob}
% G_{\ell}^{(\underline{a},N)}(\tau)=b_{0,\ell}^{(a_1,a_2)}+c_\ell\left(\sum_{\substack{m_1\equiv a_1\!\!\!\!\!\!\pmod{N}\\m_1>0}}\sum_{j=1}^\infty j^{\ell-1}\xi^{ja_2}_Nq_N^{jm_1}+(-1)^\ell\sum_{\substack{m_1\equiv -a_1\!\!\!\!\!\!\pmod{N}\\m_1>0}}\sum_{j=1}^\infty j^{\ell-1}\xi^{-ja_2}_Nq_N^{jm_1}\right)    
% \end{align}
% where $c_\ell=\dfrac{(-2\pi i)^\ell}{(\ell-1)!}$, $\xi_N=e^{2\pi i/N}$, $q_N=e^{2\pi i\tau/N}$, and 
% \begin{equation*}
% b_{0,\ell}^{(a_1,a_2)}=\begin{cases}
% 0,&a_1\neq 0,\\
% \zeta^{a_2}(\ell)+(-1)^\ell\zeta^{-a_2}(\ell),&a_1=0,
% \end{cases}    
% \end{equation*}
% where $\zeta^{a_2}(\ell)$ is defined by
% \begin{equation*}
% \zeta^{a}(\ell)=
% \sum_{\substack{n\geq 1\\n\equiv a\!\!\!\!\!\!\pmod{N}}}\dfrac{1}{n^\ell}.        
% \end{equation*}
We will need the following connection \cite{Miyake} between the Eisenstein series of the form \eqref{meis} and Koblitz' Eisenstein series of level $N$ defined by \eqref{kobeis}. 
\begin{lemma}\label{MiyKob}
Assuming the notations above, let $L\mid N$, and $R\mid N$. Furthermore, let $u, v\in\mathbb{N}$ be such that $uL\mid N$, and $vR\mid N$, then 
\begin{equation*}
M_{k,\chi,\psi}\left(\frac{u}{v}\tau\right)=v^k\sum_{\substack{0\leq a_1, a_2<N\\u\mid a_1, v\mid a_2}}\chi(a_1/u)\psi(a_2/v)G_k^{(\underline{a}, N)}(\tau).    
\end{equation*}
Equivalently, in view of \eqref{Eisensteinseries}, we have
\begin{equation*}
\mathscr{E}_{k,\chi,\psi}\left(\frac{u}{v}\tau\right)=\dfrac{v^k}{\mathcal{A}_k}\sum_{\substack{0\leq a_1, a_2<N\\u\mid a_1, v\mid a_2}}\chi(a_1/u)\psi(a_2/v)G_k^{(\underline{a}, N)}(\tau).        
\end{equation*}
\end{lemma}

By \eqref{Eisensteinseries}, both the Eisenstein series $\mathscr{E}_{2k+1,\chi_{-4},1}(2\tau)$ and $\mathscr{E}_{2k+1,\chi_{-4},\chi_2}(2\tau)$ equal zero at $\tau=i\infty$. From \eqref{transKob} and Lemma \ref{MiyKob}, we have, for $\gamma= \left (\begin{smallmatrix}
 a&b\\c&d    
 \end{smallmatrix} \right)\in SL_2(\mathbb{Z})$,
\begin{equation}\label{slashEin}
(\mathscr{E}_{k,\chi,\psi}|_{2k+1}\gamma)\left(\frac{u}{v}\tau\right)=\dfrac{v^{2k+1}}{\mathcal{A}_{2k+1}}\sum_{\substack{0\leq a_1, a_2<N\\u\mid a_1, v\mid a_2}}\chi(a_1/u)\psi(a_2/v)G_{2k+1}^{(\underline{a}\gamma, N)}(\tau).
%\quad\gamma=\begin{pmatrix}
%a&b\\c&d    
%\end{pmatrix}\in SL_2(\mathbb{Z}).     
\end{equation}
Then \eqref{slashEin} allows us to calculate, for each of the matrices in \eqref{matrices},
\begin{align}\label{Sgamma}
(\mathscr{E}_{k,\chi,\psi}|_{2k+1}S)\left(\frac{u}{v}\tau\right)&=\dfrac{v^{2k+1}}{\mathcal{A}_{2k+1}}\sum_{\substack{0\leq a_1, a_2<8\\u\mid a_1, v\mid a_2}}\chi(a_1/u)\psi(a_2/v)G_{2k+1}^{((a_2,8-a_1), 8)}(\tau),\notag\\
(\mathscr{E}_{k,\chi,\psi}|_{2k+1}\gamma_2)\left(\frac{u}{v}\tau\right)&=\dfrac{v^{2k+1}}{\mathcal{A}_{2k+1}}\sum_{\substack{0\leq a_1, a_2<8\\u\mid a_1, v\mid a_2}}\chi(a_1/u)\psi(a_2/v)G_{2k+1}^{((a_1+2a_2\!\!\!\!\!\!\pmod{8},a_2), 8)}(\tau),\notag\\
(\mathscr{E}_{k,\chi,\psi}|_{2k+1}\gamma_4)\left(\frac{u}{v}\tau\right)&=\dfrac{v^{2k+1}}{\mathcal{A}_{2k+1}}\sum_{\substack{0\leq a_1, a_2<8\\u\mid a_1, v\mid a_2}}\chi(a_1/u)\psi(a_2/v)G_{2k+1}^{((a_1+4a_2\!\!\!\!\!\!\pmod{8},a_2), 8)}(\tau).
\end{align}
Using \eqref{Sgamma}, we obtain the values of the two Eisenstein series above at the cusps $\tau=0, 1/2, 1/4$ indicated in the table below.
\vspace{0.5cm}

\begin{center}
\begin{tabular}{ |c|c|c|c|c| } 
\hline

\backslashbox{Eisenstein series}{Cusps} & 0 & $1/2$ & $1/4$ \\
\hline
% $E_{2k+1,1,\chi_{-4}}(\tau)$ & $0$ & $\beta(2k+1)$ & $\delta(2k+1)$\\ 
\hline
$\mathscr{E}_{2k+1,\chi_{-4},1}(2\tau)$ & $\frac{(-i)^{2k+1}(-1)^kE_{2k}}{2^{6k+4}}$ & $\frac{(-i)^{2k+1}(-1)^kE_{2k}}{2^{4k+3}}$ & $0$\\ 
% \hline
% $E_{2k+1,\chi_2,\chi_{-4}}(\tau)$ & $0$ & $0$ & $0$\\ 
\hline
$\mathscr{E}_{2k+1,\chi_{-4},\chi_2}(2\tau)$ & $0$ & $\frac{(-i)^{2k+1}(-1)^kE_{2k}}{2^{2k+2}}$ &$0$\\ 
\hline
\end{tabular}
\end{center}
\vspace{0.5cm}
Taken together, the cuspidal values given in \eqref{feval1}-\eqref{feval3} and the evaluations of the Eisenstein series in the table above imply \eqref{cuspform}.

\section{Proof of Theorem \ref{main2}} \label{s3}
Note that $1$ is a primitive Dirichlet character with modulus and conductor $1$, and $\chi_2$ is a Dirichlet character with modulus $2$ and conductor $1$. Thus, \eqref{meis} yields the Fourier expansion
\begin{align}\label{EiEu1}
    \mathscr{E}_{2k+1,\chi_{-4},1}(2\tau)&=\sum_{n\geq 1}\left(\sum_{d\mid n}\chi_{-4}(n/d)d^{2k}\right)q^{2n}=\sum_{n\geq 1}\sum_{m\geq 1}\chi_{-4}(n)m^{2k}q^{2mn}\notag\\
    &=\sum_{n\geq 1}\chi_{-4}(n)\sum_{m\geq 1}m^{2k}(q^{2n})^m,
\end{align}
and
\begin{align}\label{Eis2}
\mathscr{E}_{2k+1,\chi_{-4},\chi_2}(2\tau)=2^{2k+1}\mathscr{E}_{2k+1,\chi_{-4},1}(2\tau)-\mathscr{E}_{2k+1,\chi_{-4},1}(\tau).    
\end{align}
By \eqref{Eis2},
\begin{align}\label{EiEu2}
&2^{2k}\mathscr{E}_{2k+1,\chi_{-4},1}(2\tau)- \mathscr{E}_{2k+1,\chi_{-4},\chi_2}(2\tau)=\mathscr{E}_{2k+1,\chi_{-4},1}(\tau)-2^{2k}\mathscr{E}_{2k+1,\chi_{-4},1}(2\tau)\notag\\
&\hspace{2cm}=\sum_{n\geq 1}\chi_{-4}(n)\sum_{m\geq 1}m^{2k}(q^{n})^m-2^{2k}\cdot \sum_{n\geq 1}\chi_{-4}(n)\sum_{m\geq 1}m^{2k}(q^{2n})^m.
% &\hspace{2cm}=2\cdot\sum_{n\geq 1}\chi_{-4}(n)\sum_{m\geq 1}m^{2k}(q^{2n})^m-\dfrac{1}{2^{2k}}\cdot\sum_{n\geq 1}\chi_{-4}(n)\sum_{m\geq 1}m^{2k}(q^{n})^m.
% &\hspace{2cm}=2^{2k+1}\cdot \sum_{n\geq 1}\chi_{-4}(n)\sum_{m\geq 1}m^{2k}(q^{4n})^m-\sum_{n\geq 1}\chi_{-4}(n)\sum_{m\geq 1}m^{2k}(q^{2n})^m
\end{align}
Next, apply the generating function for Eulerian polynomials (see \cite{DF, FH})
\begin{equation}\label{Eulerianp}
 \sum_{\ell\geq 1}\ell^r t^\ell=\dfrac{t A_r(t)}{(1-t)^{r+1}}  
 \end{equation}
to the inner sums in \eqref{EiEu1} and \eqref{EiEu2} to obtain
\begin{align}\label{qa1}
\mathscr{E}_{2k+1,\chi_{-4},1}(2\tau)=\sum_{n\geq 1}\chi_{-4}(n)\dfrac{q^{2n} A_{2k}(q^{2n})}{(1-q^{2n})^{2k+1}},   
\end{align}
and 
\begin{align}\label{qa2}
2^{2k}\mathscr{E}_{2k+1,\chi_{-4},1}(2\tau)- \mathscr{E}_{2k+1,\chi_{-4},\chi_2}(2\tau)&=\sum_{n\geq 1}\chi_{-4}(n)\dfrac{q^{n}A_{2k}(q^{n})}{(1-q^{n})^{2k+1}}-2^{2k}\cdot\sum_{n\geq 1}\chi_{-4}(n)\dfrac{q^{2n} A_{2k}(q^{2n})}{(1-q^{2n})^{2k+1}}\notag\\
&=\sum_{n\geq 1}\chi_{-4}(n)\dfrac{q^{n}P_{k}(q^{n})}{(1-q^{2n})^{2k+1}},
\end{align}
where $P_{k}(t)\in\mathbb{Z}[t]$ is of degree $4k$, and defined by 
\begin{equation*}
P_k(t):=(1+t)^{2k+1}\cdot A_{2k}(t)-2^{2k}\cdot t\cdot A_{2k}(t^2).    
\end{equation*}
Theorem \ref{main1} and identities \eqref{qa1} and \eqref{qa2} imply \eqref{qana}. We next show identities \eqref{qana} are $q$-analogues of \eqref{betagenvalues}. For any $n\in\mathbb{N}$, by L'Hospital rule,
\begin{equation*}
\dfrac{4n^2}{4n^2-1}=\lim_{q\rightarrow 1^-}\dfrac{(1-q^{2n})^2}{(1-q^{2n+1})(1-q^{2n-1})}=\lim_{q\rightarrow 1^-}(1-q)\dfrac{(1-q^{2n})^2}{(1-q^{2n-1})^2}=\lim_{q\rightarrow 1^-}(1-q)\dfrac{(1-q^{2n})^4}{(1-q^{n})^2},    
\end{equation*}
which yields the following equivalent form of Wallis' formula:
\begin{align}\label{Wallisequiv}
\dfrac{\pi}{2}=\lim_{q\rightarrow 1^-}(1-q)\prod_{n=1}^\infty\dfrac{(1-q^{2n})^4}{(1-q^n)^2}=\lim_{q\rightarrow 1^-}(1-q)\dfrac{(q^2;q^2)_\infty^4}{(q;q)_\infty^2}.    
\end{align}
Using \eqref{Wallisequiv}, we find
\begin{align}\label{etaval}
\lim_{q\rightarrow 1^-}(1-q)^{2k+1}f_k(\tau)&=\lim_{q\rightarrow 1^-}(1-q)^{2k+1} q^{k+1}\dfrac{(q^4;q^4)_\infty^{8k-2}(q^8;q^8)_\infty^4}{(q^2;q^2)_\infty^{4k}}\notag\\
&=\left(\lim_{q\rightarrow 1^-}(1-q)\dfrac{(q^4;q^4)_\infty^{4}}{(q^2;q^2)_\infty^{2}}\right)^{2k}\cdot \lim_{q\rightarrow 1^-}(1-q)\dfrac{(q^8;q^8)_\infty^4}{(q^4;q^4)_\infty^2}\notag\\
&=\left(\lim_{q^2\rightarrow 1^-}\dfrac{(1-q^2)}{(1+q)}\dfrac{(q^4;q^4)_\infty^{4}}{(q^2;q^2)_\infty^{2}}\right)^{2k}\cdot \lim_{q^4\rightarrow 1^-}\dfrac{(1-q^4)}{(1+q+q^2+q^3)}\dfrac{(q^8;q^8)_\infty^4}{(q^4;q^4)_\infty^2}\notag\\
&=\dfrac{1}{2^{2k}}\cdot\dfrac{1}{2^2}\cdot\dfrac{\pi^{2k}}{2^{2k}}\cdot \dfrac{\pi}{2}=\dfrac{\pi^{2k+1}}{2^{4k+3}}.
\end{align}
This establishes the limits of the eta quotients on the left side of \eqref{qana}. To evaluate the limits of the Lambert series on the right side of \ref{qana}, consider two cases corresponding to $k$ even and odd.\\

\textit{Case 1}: $k$ is even. Note that $P_k(1)=2^{2k}A_{2k}(1)=2^{2k}(2k)!$. From \eqref{qana} and L'Hospital rule,
\begin{align}\label{case1}
&\lim_{q\rightarrow 1^-}\dfrac{(1-q)^{2k+1}}{2^{2k}E_{2k}}\sum_{n\geq 1}\chi_{-4}(n)\dfrac{q^{n}P_k(q^n)}{(1-q^{2n})^{2k+1}}\notag\\
&\hspace{4cm}=\dfrac{1}{2^{2k}E_{2k}}\sum_{n\geq 1}\dfrac{\chi_{-4}(n)P_{k}(1)}{(2n)^{2k+1}}=\dfrac{(2k)!\beta(2k+1)}{2^{2k+1}E_{2k}}.
\end{align}

\textit{Case 2}: $k$ is odd. Note that $A_{2k}(1)=(2k)!$. Thus, from \eqref{qana} and L'Hospital rule,
\begin{align}\label{case2}
&-\lim_{q\rightarrow 1^-}\dfrac{(1-q)^{2k+1}}{E_{2k}}\sum_{n\geq 1}\chi_{-4}(n)\dfrac{q^{2n}A_{2k}(q^{2n})}{(1-q^{2n})^{2k+1}}\notag\\
&\hspace{4cm}=-\dfrac{1}{E_{2k}}\sum_{n\geq 1}\dfrac{\chi_{-4}(n)A_{2k}(1)}{(2n)^{2k+1}}=-\dfrac{(2k)!\beta(2k+1)}{2^{2k+1}E_{2k}}.
\end{align}
By the definition of a cusp form, $T_{2k+1}(\tau)\rightarrow 0$ as $q\rightarrow 1^-$. Therefore, \eqref{etaval}, \eqref{case1} and \eqref{case2} imply \eqref{betagenvalues}. Thus, the identities given by \eqref{qana} are $q$-analogues of \eqref{betagenvalues}.

\section{Proof of Theorem \ref{main3}} \label{s4}

From Theorem \ref{GHN}, it follows that each of the eta-quotients in the sets in Theorem \ref{main3} is in $M_{2k+1}(8,\chi_{-4})$. Moreover, for every $0\leq \ell\leq 2k+1$, there is a unique form in the set such that the $q$-expansion of the form starts $q^\ell+O(q^{\ell+1})$. It follows that all such forms are linearly independent. Since $\text{dim}(M_{2k+1}(8,\chi_{-4}))=2k+2$, these forms form a basis for $M_{2k+1}(8,\chi_{-4})$. Thus, every $g\in M_{2k+1}(8,\chi_{-4})$ can be uniquely written as
\begin{align}\label{lincomb}
g(\tau)=\sum_{\ell=0}^{k-1}\alpha_k(\ell)\cdot F(\tau)^\ell\theta(2\tau)^{4(k-\ell)+2}+\sum_{\ell=k}^{2k}\beta_k(\ell)\cdot F(\tau)^{2k-\ell}F(2\tau)^{\ell-k}\theta(2\tau)^2+\gamma_k\cdot \dfrac{F(\tau)F(2\tau)^k}{\theta(2\tau)^2},  
\end{align}
where $\alpha_k(r), \beta_k(s)\in\mathbb{C}$ and $0\leq r\leq k-1,\;k\leq s\leq 2k$. If $g$ is a cusp form, then $g$ vanishes at each of the cusps in $\{0,1/2,1/4,i\infty\}$. Moreover, we have
\begin{align*}
&\hspace{3.6cm}F(i\infty)=F_2(i\infty)=0,\qquad \theta(i\infty)=\theta_2(i\infty)=1,\notag\\
&(F\mid_2 S)(i\infty)=-\dfrac{1}{2^6}, \;\;(F_2\mid_2 S)(i\infty)=-\dfrac{1}{2^8},\;\;(\theta\mid_{1/2}S)(i\infty)=\dfrac{e^{-\pi i/4}}{\sqrt{2}}, \;\;(\theta_2\mid_{1/2}S)(i\infty)=\dfrac{e^{-\pi i/4}}{2},\notag\\
&(F\mid_{2}\gamma_2)(i\infty)=\dfrac{1}{2^4},\;\; (F_2\mid_2 \gamma_2)(i\infty)=-\dfrac{1}{2^6},\;\; (\theta\mid_{1/2}\gamma_2)(i\infty)=0, \;\;(\theta_2\mid_{1/2}\gamma_2)(i\infty)=\dfrac{e^{-\pi i/4}}{\sqrt{2}},\notag\\
&\hspace{1.3cm}(F\mid_{\gamma_4})(i\infty)=0,\;\;(F_2\mid_{\gamma_4})(i\infty)=\dfrac{1}{2^4},\;\; (\theta\mid_{\gamma_4})(i\infty)=1, \;\;(\theta_2\mid_{\gamma_4})(i\infty)=0,
\end{align*}
where $S, \gamma_2$ and $\gamma_4$ are the matrices in \eqref{matrices}. The order of vanishing \cite{Ono} of $F(\tau)$, and $\theta(2\tau)^2$ at $\tau=1/4$ are given by
\begin{equation*}
\text{ord}_{1/4}(F)=\text{ord}_{1/4}(\theta_2^2)=1.     
\end{equation*} This, together with the above values of $F(\tau), F(2\tau), \theta(\tau)$ and $\theta(2\tau)$ at the cusps $\{0, 1/2, 1/4, i\infty\}$,
imply that if $g$ is a cusp form then the coefficients $\alpha_k(r), \beta_k(s)$ and $\gamma_k$ in \eqref{lincomb} satisfy the conditions (1)-(3) in Theorem \ref{main3}.

\section{CM Evaluations for eta quotients and a proof of Theorem \ref{main4}}\label{CMeva} \label{s5}
%A standard way to evaluate such eta quotients is through the following Theorem.
 % For the purpose of CM evaluations of modular forms, we introduce some notations and conventions. 
The standard way to evaluate eta quotients associated with CM values for imaginay quadratic fields with fundamental discriminants is the Chowla-Selberg formula. Let $\overline{\mathbb{Q}}$ be the algebraic closure of the field of rational numbers, and suppose $D<0$ is the fundamental discriminant of $\mathbb{Q}(\sqrt{D})$. Let $h(D)$ denote the class number of $\mathbb{Q}(\sqrt{D})$, and define $h'(D):=1/3$ (resp. $1/2$) when $D=-3$ (resp. $-4$), and $h'(D):=h(D)$ when $D<4$. We then let
 \begin{align*}
\Omega_D:=\dfrac{1}{\sqrt{2\pi|D|}}\left(\prod_{j=1}^{|D|-1}\Gamma\left(\dfrac{j}{|D|}\right)^{\chi_D(j)}\right)^{\frac{1}{h'(D)}},    
 \end{align*}
 where $\chi_D:=\left(\frac{D}{\cdot}\right)$. Then the Chowla-Selberg formula \cite{CS, CS1, lerch1897quelques, MR215797} 
 (cf. \cite[Theorem 9.3]{VW}) yields
 \begin{theorem}\label{CS}
 Under the notations above, we have
 $f(\tau)\in\overline{\mathbb{Q}}\cdot\Omega_D^k$    
 for all $\tau\in\mathbb{H}\cup\mathbb{Q}(\sqrt{D})$, all $k\in\mathbb{Z}$, and all modular forms $f$ of weight $k$ with algebraic Fourier coefficients. 
 \end{theorem}
Specializing Theorem \ref{CS} to evaluate eta quotients at CM requires a great deal of supplementary data. Theorem \ref{main4} results from our more elementary approach to evaluate the Dedekind eta function at CM points. We evaluate $\eta(\tau)$ explicitly for a special class of quadratic irrationals $\tau$ using an analytic method. This will lead to CM evaluations of the identities in Theorem \ref{main2}.
\begin{lemma}\label{eta2k}
For $k\in\mathbb{N}$, we have
\begin{align*}
\eta(2^k i)=\dfrac{\pi^{1/4}}{\Gamma(3/4)\cdot 2^{\frac{k+1}{2}}}\cdot e^{-\frac{\pi(2^{k}-1)^2}{12\cdot 2^k}-\frac{1}{2\pi}\sum_{\ell=1}^k 2^\ell L_\ell},    
\end{align*}
where $L_\ell$ is defined as in Theorem \ref{main4} and satisfies the evaluations appearing in \eqref{ev}.
% by convergent double series  
% \begin{align*}
% L_\ell:=\sum_{m,n\geq 1}\dfrac{(-1)^m}{n^2+2^\ell m^2},    
% \end{align*}
% and we have
% \begin{align*}
% \sum_{m,n\geq 1}\dfrac{(-1)^m}{n^2+2^\ell m^2}=\begin{cases}
% -\dfrac{\pi^2}{24}-\dfrac{\pi\cdot\log 2}{8},&\ell=1,\\ -\dfrac{7\pi^2}{96}-\dfrac{\pi\cdot\log 2}{32}-\dfrac{\pi\cdot\log(\sqrt{2}-1)}{8},&\ell=2,\\
% -\dfrac{31\pi^2}{384}-\dfrac{5\pi\log 2}{128}+\dfrac{\pi\cdot\log(\sqrt{2}-1)}{32}-\dfrac{\pi\cdot\log(1-2^{-1/4})}{8},&\ell=3.
% \end{cases}    
% \end{align*}
\end{lemma}
\begin{proof}
Note that 
\begin{align}\label{etaeval}
\eta(2^k i)&=e^{-2^{k+1}\pi/24}\prod_{\ell=1}^\infty(1-e^{-2^{k+1}\pi\ell})=e^{-2^{k+1}\pi/24}\prod_{\ell=1}^\infty(1-e^{-2^{k}\pi\ell})\prod_{\ell=1}^\infty(1+e^{-2^{k}\pi\ell})\notag\\
&=:e^{-2^k\pi/24}\cdot g_k\cdot \eta(2^{k-1}i),  
\end{align}
where 
\begin{align*}
g_k:=\prod_{\ell=1}^\infty(1+e^{-2^{k}\pi\ell}).
% =\dfrac{1}{2}\prod_{\ell=0}^\infty(1+e^{-2^{k}\pi\ell}).    
\end{align*}
We next evaluate $g_k$. Taking logarithm of $g_k$, then using Stieltjes integration and integration by parts, we obtain
\begin{align}\label{int}
\log g_k&=\sum_{\ell=1}^\infty\log(1+e^{-2^{k}\pi\ell})=\int_{1^{-}}^{\infty}\log(1+e^{-2^{k}\pi t})d[t]\notag\\
&=\left.[t]\log(1+e^{-2^{k}\pi t})\right|_{1^-}^\infty+2^k\pi\int_1^{\infty}\dfrac{[t]\cdot e^{-2^k\pi t}}{1+e^{-2^{k}\pi t}}dt\notag\\
&=2^k\pi\int_0^{\infty}\dfrac{[t]}{1+e^{2^{k}\pi t}}dt\quad\quad (\text{since $[t]=0$ for $0\leq t\leq 1$})\notag\\
&=2^k\pi\left(\int_0^\infty\dfrac{t}{1+e^{2^{k}\pi t}}dt- \int_0^\infty\dfrac{\{t\}}{1+e^{2^{k}\pi t}}dt\right).
\end{align}
Let us denote the first and the second integrals inside parentheses above by $I_1(k)$ and $I_2(k)$, respectively. By the residue theorem, it follows that 
\begin{equation}\label{I1k}
I_1(k)=\dfrac{1}{48\cdot 4^{k-1}}.    
\end{equation}
To evaluate $I_2(k)$, we recall the Fourier series for $\{t\}$ as follows:
\begin{equation*}
\{t\}=\dfrac{1}{2}-\sum_{n\geq 1}\dfrac{\sin(2\pi nt)}{\pi n},    
\end{equation*}
which yields
\begin{align}\label{I2k}
I_2(k)&=\dfrac{1}{2}\int_0^\infty\dfrac{dt}{1+e^{2^{k}\pi t}}-\dfrac{1}{\pi}\sum_{n\geq 1}\dfrac{1}{n}\int_0^\infty\dfrac{\sin(2\pi nt)}{1+e^{2^{k}\pi t}}dt\notag\\
&=-\dfrac{1}{2^{k+1}\pi}\int_0^\infty d(\log(1+e^{-2^{k}\pi t}))-\dfrac{1}{\pi}\sum_{n\geq 1}\sum_{m\geq 0}\dfrac{(-1)^m}{n}\int_0^\infty\sin(2\pi nt)e^{-2^k\pi(m+1)t}dt\notag\\
&=\dfrac{\log 2}{2^{k+1}\pi}-\dfrac{1}{\pi}\sum_{n\geq 1}\sum_{m\geq 0}\dfrac{(-1)^m}{n}\int_0^\infty\sin(2\pi nt)e^{-2^k\pi(m+1)t}dt.
\end{align}
Next, we evaluate the integral in the right-hand side of \eqref{I2k}. We have
\begin{align}\label{exint}
\int_0^\infty e^{2\pi i nt}e^{-2^k\pi(m+1)t}dt&=\int_0^\infty e^{2\left(\pi i n-2^{k-1}\pi(m+1)\right)t}dt=-\dfrac{1}{2\pi(i n-2^{k-1}(m+1))}\notag\\
\int_0^\infty e^{-2\pi i nt}e^{-2^k\pi(m+1)t}dt&=\int_0^\infty e^{-2\left(\pi i n+2^{k-1}\pi(m+1)\right)t}dt=\dfrac{1}{2\pi(i n+2^{k-1}(m+1))}
\end{align}
Using \eqref{exint}, it now follows that
\begin{align}\label{sineint}
\int_0^\infty\sin(2\pi nt)e^{-2^k\pi(m+1)t}dt&=\dfrac{1}{2i}\left(\int_0^\infty e^{2\pi i nt}e^{-2^k\pi(m+1)t}dt-\int_0^\infty e^{-2\pi i nt}e^{-2^k\pi(m+1)t}dt\right)\notag\\
&=-\dfrac{1}{2i}\left(\dfrac{1}{2\pi(i n-2^{k-1}(m+1))}+\dfrac{1}{2\pi(i n+2^{k-1}(m+1))}\right)\notag\\
&=\dfrac{1}{2\pi}\left(\dfrac{n}{n^2+2^k(m+1)^2}\right).
\end{align}
Combining \eqref{I2k} and \eqref{sineint}, we get
\begin{align}\label{I2kf}
I_2(k)=\dfrac{\log 2}{2^{k+1}\pi}+\dfrac{1}{2\pi^2}\sum_{m,n\geq 1}\dfrac{(-1)^m}{2^km^2+n^2}.    
\end{align}
Finally, from \eqref{int}, \eqref{I1k} and \eqref{I2kf}, we arrive at
\begin{align}\label{fgk}
\log g_k=-\dfrac{\log 2}{2}+\dfrac{\pi}{3\cdot 2^{k+2}}-\dfrac{2^{k-1}}{\pi}\sum_{m,n\geq 1}\dfrac{(-1)^m}{2^km^2+n^2}.    
\end{align}
Let us denote the series in the right-hand side of \eqref{fgk} by $L_k$. Taking the exponential on both sides of \eqref{fgk} yields
\begin{align}\label{gkf}
 g_k=2^{-1/2}\cdot e^{\frac{\pi}{3\cdot 2^{k+2}}-\frac{2^{k-1}}{\pi}\cdot L_k}.   
\end{align}
Combining \eqref{etaeval}, \eqref{gkf} and noting that by \cite[Entry 2 (p. 256) (ii), page 326]{MR1486573}, $\eta(i)=\dfrac{\pi^{1/4}}{\Gamma(3/4)\cdot 2^{1/2}}$.  Iteration results in the following:
\begin{align}\label{eta2ki}
\eta(2^k i)=e^{-\pi(2^k+2^{k-1}+\cdots+2)/24}\eta(i)\prod_{\ell=1}^kg_\ell &=\dfrac{\pi^{1/4}}{\Gamma(3/4)\cdot 2^{1/2}}\cdot e^{-\frac{\pi(2^{k}-1)}{12}}\cdot 2^{-k/2}\cdot e^{\frac{\pi}{12}\left(1-\frac{1}{2^{k}}\right)-\frac{1}{2\pi}\sum_{\ell=1}^k 2^\ell L_\ell}\notag\\
&=\dfrac{\pi^{1/4}}{\Gamma(3/4)\cdot 2^{\frac{k+1}{2}}}\cdot e^{-\frac{\pi(2^{k}-1)^2}{12\cdot 2^k}-\frac{1}{2\pi}\sum_{\ell=1}^k 2^\ell L_\ell}.
\end{align}
From the Jacobi triple product identity \cite[Theorem 1.3.3]{B},
\begin{align*}
\Psi(\tau):=\sum_{n\geq 0}q^{n(n+1)/2}=e^{-\frac{2\pi i\tau}{8}}\cdot\dfrac{\eta(2\tau)^2}{\eta(\tau)}\quad\text{and}\quad \tilde{\theta}(\tau):=\sum_{n=-\infty}^\infty (-1)^n q^{n^2}=\dfrac{\eta(\tau)^2}{\eta(2\tau)}.  
\end{align*}    
From \cite[p.~326]{MR1486573},
%\cite[Entry 2 (p. 256), page 326]{MR1486573}, 
we have
\begin{align}\label{valuese}
% \eta(i)=\dfrac{\pi^{1/4}}{\Gamma(3/4)\cdot 2^{1/2}},\quad 
\eta(2i)=\dfrac{\pi^{1/4}}{\Gamma(3/4)\cdot 2^{7/8}},\quad \eta(4i)=\dfrac{\pi^{1/4}\cdot(\sqrt{2}-1)^{1/4}}{\Gamma(3/4)\cdot 2^{21/16}},\end{align}
and
\begin{align}\label{valuespsi}
\Psi(2i)=\dfrac{\pi^{1/4}(2-2^{1/2})^{1/2}\cdot e^{\pi/2}}{\Gamma(3/4)\cdot 2^2},\quad \Psi(4i)=\dfrac{\pi^{1/4}(1-2^{-1/4})\cdot e^{\pi}}{\Gamma(3/4)\cdot 2^2}.
\end{align}
Next, from \cite[page 115, Entry 8 (xi)-(xii)]{B1}, we obtain 
\begin{equation*}
\eta(\tau)\cdot\eta(2\tau)=q^{1/8}\cdot\tilde{\theta}(\tau)\cdot \Psi(\tau),\quad\quad 
\dfrac{\eta(4\tau)}{\eta(\tau)}=e^{\frac{2\pi i\tau}{8}}\cdot\dfrac{\Psi(\tau)}{\tilde{\theta}(2\tau)}.    
\end{equation*}
Applying this results in
\begin{equation}\label{rec2}
\eta(4\tau)^2=e^{3\pi i\tau/4}\cdot\dfrac{\eta(\tau)\cdot\Psi(\tau)\cdot\Psi(2\tau)}{\eta(2\tau)}.
\end{equation}
Taking $\tau=2i$ in \eqref{rec2}, and using \eqref{valuese} and \eqref{valuespsi}, we obtain
\begin{align}\label{eta8}
\eta(8i)=\dfrac{\pi^{1/4}\cdot (\sqrt{2}-1)^{1/8}\cdot (1-2^{-1/4})^{1/2}}{\Gamma(3/4)\cdot 2^{53/32}}.
\end{align}
Comparing the values of $\eta(2i), \eta(4i)$ in \eqref{valuese}, and $\eta(8i)$ in \eqref{eta8} with those obtained from \eqref{eta2ki}, we arrive at the result.
\end{proof}

\subsection{Proof of Theorem \ref{main4}}
For $r\in\mathbb{N}$, consider the CM point $\tau_r:=2^r i$. 
% We note that $\omega_D^{2k+1}=2^{\frac{2k+1}{2}}\cdot |D|^{\frac{2k+1}{2}}\cdot \Omega_D^{2k+1}$, and that $2^{\frac{2k+1}{2}}\cdot |D|^{\frac{2k+1}{2}}\in\overline{\mathbb{Q}}$. It now follows using Theorem \ref{CS} that $H_k(\tau)\in\overline{\mathbb{Q}}\cdot\omega_{D}^{2k+1}$ for $\tau\in\mathbb{H}\cap\mathbb{Q}(\sqrt{D})$ where $D$ is a fundamental discriminant of $\mathbb{Q}(\sqrt{D})$.
% Next, we note the following important formula from \cite[Entry 8 (xii)]{RN2} expressed in terms of the Dedekind eta-function, and our notations:
% \begin{equation*}
% q^{1/8}\cdot\dfrac{\eta(\tau)}{\eta(4\tau)}=\dfrac{\theta(2\tau)}{\Psi(\tau)}\Rightarrow \eta(4\tau)=q^{1/8}\cdot\dfrac{\Psi(\tau)}{\theta(2\tau)}\cdot\eta(\tau)    
% \end{equation*}
% let $\tau$ be the special CM point in $\{\}$
Using Theorem \ref{main1}, and conditions (1) and (2) in Theorem \ref{main3}, we see that
\begin{equation}\label{Hktr}
H_k(\tau)=f_k(\tau)+\sum_{\ell=1}^{k-1}\alpha_k(\ell)\cdot F(\tau)^\ell\theta(2\tau)^{4(k-\ell)+2}+\sum_{\ell=k}^{2k}\beta_k(\ell)\cdot F(\tau)^{2k-\ell}F(2\tau)^{\ell-k}\theta(2\tau)^2   
\end{equation}
with uniquely determined coefficients $\alpha_k(r), \beta_k(s)\in\mathbb{C}$ for $1\leq r\leq k-1$ and $k\leq s\leq 2k$ satisfying conditions (2) and (3) in Theorem \ref{main3}. Set $a=\frac{\pi^{1/4}}{\Gamma(3/4)}$. Then noting that 
\begin{align*}
F(\tau)=\dfrac{\eta(4\tau)^8}{\eta(2\tau)^4},\quad F(2\tau)=\dfrac{\eta(8\tau)^8}{\eta(4\tau)^4},\quad \theta(2\tau)=\dfrac{\eta(4\tau)^5}{\eta(2\tau)^2\eta(8\tau)^2},   
\end{align*}
and using Lemma \eqref{eta2k}, we find
\begin{align}\label{Fther}
F(\tau_r)&=\dfrac{\eta(2^{r+2}i)^8}{\eta(2^{r+1}i)^4}=\dfrac{a^4}{2^{2r+8}}\cdot e^{-\frac{2\pi(2^{r+1}+2^r-1)}{3}-\frac{2}{\pi}\sum_{m=1}^{r+1}2^m L_m-\frac{2^{r+4}}{\pi}L_{r+2}},\notag\\
F(2\tau_r)&=\dfrac{\eta(2^{r+3}i)^8}{\eta(2^{r+2}i)^4}=\dfrac{a^4}{2^{2r+10}}\cdot e^{-\frac{2\pi(2^{r+2}+2^{r+1}-1)}{3}-\frac{2}{\pi}\sum_{m=1}^{r+2}2^m L_m-\frac{2^{r+5}}{\pi}L_{r+3}},\notag\\
\theta(2\tau_r)&=\dfrac{\eta(2^{r+2}i)^5}{\eta(2^{r+1}i)^2\eta(2^{r+3}i)^2}=\dfrac{a}{2^{\frac{r+3}{2}}}\cdot e^{\frac{\pi}{6}-\frac{1}{2\pi}\sum_{m=1}^{r+1}2^m L_m-\frac{3\cdot 2^{r+1}}{\pi}L_{r+2}+\frac{2^{r+3}}{\pi}L_{r+3}}.
\end{align}
We also have
\begin{align}\label{fktr}
f_k(\tau_r)=\dfrac{\eta(2^{r+2}i)^{8k-2}\eta(2^{r+3}i)^4}{\eta(2^{r+1}i)^{4k}}=\dfrac{a^{4k+2}}{2^{2k(r+4)+r+5}}\cdot e^{\frac{\pi(2k+1)}{3}-2^{r+1}(k+1)\pi-\frac{2k+1}{\pi}\sum_{m=1}^{r+1}2^m L_m-\frac{2^{r+2}(4k+1)}{\pi}L_{r+2}-\frac{2^{r+4}}{\pi}L_{r+3}}.    
\end{align}
From \eqref{Hktr}, \eqref{Fther} and \eqref{fktr}, we obtain
\begin{align*}
H_k(\tau_r)&=\dfrac{a^{4k+2}}{2^{(2k+1)(r+3)}}\cdot e^{J(k,r)}\left(\dfrac{e^{-2^{r+1}(k+1)\pi}}{4^{k+1}}+\sum_{\ell=1}^{k-1}\dfrac{\alpha_k(\ell)}{4^\ell}\cdot e^{-\frac{\ell\cdot 2^{r+1}}{\pi}\sum_{m=1}^{r+1}2^m L_m-\frac{2^{r+3}(k-\ell+1)}{\pi}L_{r+2}+\frac{2^{r+5}(k-\ell+1)}{\pi}L_{r+3}}\right.\notag\\
&\hspace{5cm}+\left.\sum_{\ell=k}^{2k}\dfrac{\beta_k(\ell)}{4^\ell}\cdot e^{-\frac{\ell\cdot 2^{r+1}}{\pi}\sum_{m=1}^{r+1}2^m L_m-\frac{2^{r+3}(k-\ell+1)}{\pi}L_{r+2}+\frac{2^{r+5}(k-\ell+1)}{\pi}L_{r+3}}\right)
\end{align*}
where 
\begin{align*}
J(k,r):=\frac{\pi(2k+1)}{3}-\frac{2k+1}{\pi}\sum_{m=1}^{r+1}2^m L_m-\frac{2^{r+2}(4k+1)}{\pi}L_{r+2}-\frac{2^{r+4}}{\pi}L_{r+3}.   
\end{align*}
This completes the proof of Theorem \ref{main4}.

\section{A Combinatorial Interpretation with Asymptotics} \label{s6}
We conclude the paper with a combinatorial interpretation for the coefficients of the eta function $f_{k}(\tau)$ and an asymptotic formula that follows as a result of Theorem \ref{main1}. Let $\Psi(q)$ denote the generating function for triangular numbers, defined by
\begin{equation}\label{gentriangular}
\Psi(\tau)=\Psi(q):=\sum_{n\geq 0}q^{\frac{n(n+1)}{2}}=\sum_{n\geq 0}q^{T_n}.    
\end{equation}
For a fixed tuple $(a_1, a_2,\cdots,a_k)\in\mathbb{N}^k$, let $$t_{k,a_1,a_2,\cdots,a_k}(n):=|\{(n_1,n_2,\cdots,n_{k})\in\mathbb{N}^{k}: n=a_1\cdot T_{n_1}+a_2\cdot T_{n_2}+\cdots+T_{n_{k}})\}|$$ be the number of representations of $n$ as a linear combination of $k$ triangular numbers with coefficients $a_1, a_2,\cdots, a_k$. 
Denote $$t_k(n):=|\{(n_1,n_2,\cdots,n_{4k},a,b)\in\mathbb{N}^{4k+2}: n=T_{n_1}+T_{n_2}+\cdots+T_{n_{4k}}+2(T_{a}+T_{b})\}|.$$ Then $t_k(n)=t_{4k+2,1,1,\cdots,1,2,2}(n)$  (with $4k$ $1$'s) counts the number of representations of $n$ as a sum of $4k$ triangular numbers and twice the sum of two triangular numbers. In \cite{ORW}, the function $t_{4k,1,1,\cdots,1}(n)$ (with $4k$ $1$'s) was studied which counts the number of representations of $n$ as sums of $4k$ triangular numbers. 
%*(More references of similar studies)*. 
From the eta quotient representation \begin{align*}
\Psi(\tau)=e^{-\frac{2\pi i\tau}{8}}\cdot\dfrac{\eta(2\tau)^2}{\eta(\tau)},
%\quad\text{and}\quad \tilde{\theta}(\tau):=\sum_{n=-\infty}^\infty (-1)^n q^{n^2}=\dfrac{\eta(\tau)^2}{\eta(2\tau)}.  
\end{align*}   
we we have the following combinatorial interpretation for $f_{k}$:
\begin{equation}\label{gf}
f_k(\tau):=q^{k+1}\Psi(q^4)^2\cdot\Psi(q^2)^{4k}=:\sum_{n\geq 0}t_{k}(n)q^{2n+k+1},
\end{equation}
whose coefficients satisfy the asymptotic formula given in Corollary \ref{cor1}.
\begin{cor}\label{cor1}
For $k\geq 1$, we have
\begin{align*}
t_k(n)\sim \begin{cases}
\dfrac{1}{2^{2k}E_{2k}}\cdot\sigma_{\chi_{-4};2k}(2n+k+1),&k\equiv 0\pmod{2},\\
-\dfrac{1}{E_{2k}}\cdot \sigma_{\chi_{-4};2k}(2n+k+1),&k\equiv 1\pmod{2},
\end{cases}  
\end{align*}
where 
\begin{equation*}
\sigma_{\chi_{-4};2k}(n):=\sum_{d\mid n}\chi_{-4}(n/d)d^{2k}.    
\end{equation*}
% and
% \begin{align*}
% \sigma_{\chi_{-4};2k}^{\#}(n):=\sum_{\substack{d\mid n\\d\;\text{odd}}}\chi_{-4}(n/d)d^{2k}=\begin{cases}
% \sigma_{\chi_{-4};2k}(n)-2^{2k}\sigma_{\chi_{-4};2k}(n/2),&n\equiv 0\pmod{2},\\
% \sigma_{\chi_{-4};2k}(n),&n\equiv 1\pmod{2}.
% \end{cases}    
% \end{align*}
%Here $a(n)\sim b(n)$ means that $\lim_{n\rightarrow\infty}\frac{a(n)}{b(n)}=1$.
\end{cor}

\begin{proof}
From \eqref{gf} and Theorem \ref{main1}, it follows that
\begin{align}\label{tkn}
t_k(n)=[q^{2n+k+1}]H_k(\tau)+[q^{2n+k+1}]T_k(\tau),    
\end{align}
where $[q^\ell]f(\tau)$ denotes the $\ell$th coefficient in the power series expansion about $q=0$ of $f(\tau)$. Let $a_k(\ell)$ denote the $\ell$th Fourier coefficient of $T_{2k+1}(\tau)$. We consider two cases.\\

\textit{Case 1}: $k$ is even. In this case, we have from \eqref{Eis2} and \eqref{EiEu2} that
\begin{align}\label{coeffcase1}
H_k(\tau)&=\dfrac{1}{2^{2k}E_{2k}}\left(\mathscr{E}_{2k+1,\chi_{-4},1}(\tau)-2^{2k}\mathscr{E}_{2k+1,\chi_{-4},1}(2\tau)\right)+T_{2k+1}(\tau)\notag\\
&=\dfrac{1}{2^{2k}E_{2k}}\left(\sum_{n\geq 1}\sigma_{2k,\chi_{-4}}(n)q^n-2^{2k}\sum_{n\geq 1}\sigma_{2k,\chi_{-4}}(n)q^{2n}\right)+\sum_{n\geq 1}a_k(n)q^n.
\end{align}
Since $2n+k+1\equiv 1\pmod{2}$, \eqref{tkn} and \eqref{coeffcase1} imply that
\begin{align}\label{tkncase1}
t_k(n)=
\dfrac{1}{2^{2k}E_{2k}}\cdot\sigma_{\chi_{-4};2k}(2n+k+1)+a_k(2n+k+1).
\end{align}

\textit{Case 2}: $k$ is odd. Since $2n+k+1\equiv 0\pmod{2}$, we have from \eqref{tkn} and Theorem \ref{main1} that
\begin{align}\label{tkncase2}
t_k(n)=-\dfrac{1}{E_{2k}}\cdot\sigma_{\chi_{-4};2k}(2n+k+1)+a_k(2n+k+1).
\end{align}
From Deligne's work \cite{MR340258} (cf.~\cite{conrad2008modular}), we know that $|a_k(n)|\leq d(n)n^{k}$, and since $\sigma_{\chi_{-4};2k}(n)$ has lower bound $n^{2k-1}$, \eqref{tkncase1} and \eqref{tkncase2} yield the claimed result.
\end{proof}

\end{document}